\documentclass[preprint]{elsarticle}
\biboptions{sort&compress}
\usepackage{graphicx} 
\usepackage{hyperref}
\usepackage{amsmath, amssymb, amsfonts}
\usepackage{tikz}
\usepackage{algorithm} 
\usepackage{algpseudocode}

\newtheorem{theorem}{Theorem}[section]
\newtheorem{corollary}[theorem]{Corollary}
\newtheorem{lemma}[theorem]{Lemma}
\newtheorem{proposition}[theorem]{Proposition}

\newdefinition{definition}{Definition}
\newdefinition{remark}{Remark}

\newproof{pot}{\textnormal{\textbf{Proof}}}

\newenvironment{proof}{\begin{pot}}{\qed \end{pot}}

\begin{document}

\begin{frontmatter}
    \title{The Mutual-Visibility Problem In\\Directed Graphs}

    \author[1]{Vanja Stojanovi\'c \corref{cor1}}
    \ead{vs66277@student.uni-lj.si}
    \cortext[cor1]{Corresponding author}

    \affiliation[1]{organization={Faculty of Mathematics and Physics, University of Ljubljana},
    country={Slovenia}}

    \begin{abstract}
    The study of mutual visibility has traditionally focused on undirected graphs, asking for the maximum number of vertices that can communicate via shortest paths without intermediate interference from other set members. In this paper, we extend this concept to directed graphs, establishing fundamental results for several graph classes. We prove that for Directed Acyclic Graphs (DAGs), the mutual-visibility number $\mu(D)$ is always 1, and for directed cycles of length $n\ge3$, it is strictly 2. In contrast, we demonstrate that tournaments can support arbitrarily large mutual-visibility sets; specifically, using properties of Paley tournaments, we show that $\mu(T)$ grows linearly with the size of the tournament. On the algorithmic side, we show that while verifying a candidate set is polynomial-time solvable ($O(|S|(|V|+|A|))$), the problem of determining $\mu(D)$ is NP-hard for general digraphs. We also analyze the impact of strong bridges and strongly connected components on the upper bounds of $\mu(D)$.
    \end{abstract}
    
\end{frontmatter}

\section{Introduction}

The \emph{mutual‐visibility} problem on graphs asks for a largest set of vertices such that any two in the set are “mutually visible,” meaning there is at least one shortest path between them whose internal vertices lie outside the set. Formally, a set $S\subseteq V(G)$ is a \emph{mutual‐visibility set} if for every $x,y\in S$, $x\neq y$, there exists a shortest $x$,$y$-path in $G$ with no internal vertex in $S$.  The size of a maximum mutual‐visibility set is the \emph{mutual‐visibility number} $\mu(G)$. It was first defined in \cite{DISTEFANO2022126850} by Di Stefano.

Di Stefano's seminal work \cite{DISTEFANO2022126850} explicitly motivates further exploration into the mutual-visibility problem, particularly within the context of directed graphs. This specific area currently exhibits a notable scarcity of dedicated research. Our work aims to address this gap by establishing fundamental results concerning mutual-visibility in directed graphs.

\paragraph{Preliminaries} A \emph{graph} $G = (V, E)$ consists of a finite vertex set $V$ and an edge set $E$ of unordered pairs of distinct vertices. A \emph{directed graph} (or \emph{digraph}) $D = (V, A)$ consists of a finite vertex set $V$ and a set $A$ of ordered pairs $(u, v)$ of distinct vertices named directed edges. We discuss only simple graphs, which have no parallel edges.

When we remove edge orientations from a digraph $D$, we obtain the underlying simple undirected graph $G$ of $D$.
If $X$ is a subset of a graph or digraph $G$, then we denote the complement as $\overline X = V(G) \setminus X$.

A vertex $v$ is \emph{reachable} from $u$ in a digraph $D$, denoted $u \rightarrow v$, if there exists a path from $u$ to $v$. A \emph{path} $P = (x_0, x_1, \ldots, x_k)$ from $u=x_0$ to $v=x_k$ is a sequence of distinct vertices such that $(x_i, x_{i+1}) \in A$ for all $0 \le i < k$. The \emph{length} of $P$ is $k$. The distance from $u$ to $v$ in $D$, denoted by $d_D(u, v)$ (or simply $d(u, v)$), is the length of a shortest directed path from $u$ to $v$. If no such path exists, $d(u, v) = \infty$. The internal vertices of $P$ are $x_1, \ldots, x_{k-1}$. The set of all paths from $u$ to $v$ in $D$ is denoted by $\mathcal{P}_{D}(u, v)$. In directed graphs, the path is named a directed path. Two vertices $u$ and $v$ are \emph{mutually reachable} if $u \rightarrow v$ and $v \rightarrow u$, we denote such a relation with $u \leftrightarrow v$.

For a subset of vertices $X \subseteq V$, the subgraph induced by $X$, denoted $D[X]$, is the digraph with vertex set $X$ and edge set $A \cap (X \times X)$. We denote by $D - X$ the subgraph induced by $V \setminus X$. Similarly, for an edge $e \in A$, $D - e$ denotes the digraph $(V, A \setminus \{e\})$. A strongly connected component of a digraph is a maximal connected subgraph in which vertices are pairwise mutually reachable.

We define a complete directed graph on $n$ vertices as $K_n$, where $V(K_n) = \{v_1, ..., v_n\}$ and $A = \{(v_i, v_j)\ |\ \forall i,j\in [n]\}$

\paragraph{Related work} The concept of mutual visibility was introduced by Di Stefano \cite{DISTEFANO2022126850} in the context of undirected graphs. This work was motivated by problems in robot navigation and communication networks, where entities must maintain lines of sight or communication channels without interference from other entities in the swarm. Di Stefano established that the mutual-visibility number is related to, but distinct from, other metric invariants such as the general position number and the geodetic number.

Following this seminal work, Cicerone et al. \cite{CICERONE2023114096} expanded the field by investigating a variety of mutual-visibility problems, considering different constraints on the paths (e.g., all shortest paths vs. at least one shortest path). They provided bounds for grids, orbital graphs, and other specific topologies.

However, the literature has remained focused almost exclusively on undirected graphs. To the best of our knowledge, the problem of mutual visibility in directed graphs has not been addressed. The transition to directed graphs introduces significant structural challenges: distance is no longer symmetric ($d(u,v) \neq d(v,u)$), and the existence of a path does not guarantee a return path. These asymmetries render many undirected techniques inapplicable. For instance, while removing an edge in a cycle simply increases distance in the undirected case, it destroys strong connectivity entirely in the directed case. Our work aims to bridge this gap by defining the problem for digraphs and characterizing $\mu(D)$ for fundamental directed structures. Consider the practical implications of this asymmetry in communication networks. In an undirected model, a link between two nodes implies a bidirectional channel. However, in real-world systems like optical networks with isolators or traffic grids with one-way streets, visibility is not inherently reciprocal. A node $u$ might "see" $v$, but $v$ may require a completely different, longer route to see $u$.

Finally, we distinguish our contributions from recent parallel developments in the undirected setting. Notably, Roy et al. \cite{roy2025vertexvisibilitynumbergraphs}recently analyzed the vertex visibility number, linking the visibility of a subset $S$ from a fixed root $x$ to the leaves of a shortest-path tree. While their work provides significant insights into the structure of visibility in undirected graphs, their reliance on shortest-path trees is inherently symmetric. In contrast, our work addresses the strictly directed case, where path asymmetry ($u \to v \nleftrightarrow v \to u$) prevents the formation of standard shortest-path trees, requiring a fundamentally different structural analysis based on strong connectivity and cycle decomposition.

\section{Mutual-Visibilities and Properties in Directed Graphs}
\label{sec:mut_visibilities}

We first define mutual-visibility, adapted to directed graphs.

\begin{definition}
    In a digraph $D$ two vertices are mutually-visible, if there exists a directed shorted path between them, without further selected vertices.
\end{definition}

\noindent
Following the definitions of \cite{CICERONE2023114096}, we also define mutual-visibility sets.

\begin{definition}[Mutual-visibility set]
    \label{def:directed_mv_set}
    In a digraph $D$, a set $S\subseteq V(D)$ is a mutual-visibility set, if $\forall x,y \in S$ there exist shortest directed $x,y$-path $P$ and $y,x$-path $Q$ such that $(S \cap V(P)) \cup (S \cap V(Q)) = \{x,y\}$. The size of a maximum such set is denoted by $\mu(D) = \max_{S\subseteq D}|S|$ and known as the mutual-visibility number of $D$.
\end{definition}

\begin{remark}
    We say that vertices of a mutual-visibility set $S$ are $S$-visible.
\end{remark}

\begin{definition}[Total mutual-visibility set]
    In a digraph $D$, a set $S\subseteq V(D)$ is a total mutual-visibility set if $\forall x,y \in V(D)$ there exists shortest directed $x,y$-path $P$ and $y,x$-path $Q$ such that $(S \cap V(P)) \cup (S \cap V(Q)) = \{x,y\}$. The size of a maximum such set is denoted by $\mu_t(D) = \max_{S\subseteq D}|S|$ and known as the total mutual-visibility number of $D$.
\end{definition}

\begin{definition}[Outer mutual-visibility set]
    In a digraph $D$, a set $S\subseteq V(D)$ is an outer mutual-visibility set if $\forall x,y \in S$ there exists shortest directed $x,y$-path $P$ and $y,x$-path $Q$ such that $(S \cap V(P)) \cup (S \cap V(Q)) = \{x,y\}$, and $\forall v\in S\wedge \forall u \in \overline S$ there exists shortest directed $v,u$-path $R$ and $u,v$-path W such that $(S \cap V(R)) \cup (S \cap V(W)) = \{u,v\}$. The size of a maximum such set is denoted by $\mu_o(D) = \max_{S\subseteq D}|S|$ and known as the outer mutual-visibility number of $D$.
\end{definition}

\begin{definition}[Dual mutual-visibility set]
    In a digraph $D$, a set $S\subseteq V(D)$ is a dual mutual-visibility set if $\forall x,y \in S$ there exists shortest directed $x,y$-path $P$ and $y,x$-path $Q$ such that $(S \cap V(P)) \cup (S \cap V(Q)) = \{x,y\}$, and $\forall v,u \in \overline S$ there exists shortest directed $v,u$-path $R$ and $u,v$-path $W$ such that $(S \cap V(R)) \cup (S \cap V(W)) = \{u,v\}$. The size of a maximum such set is denoted by $\mu_d(D) = \max_{S\subseteq D}|S|$ and known as the dual mutual-visibility number of $D$.
\end{definition}

For further notation, let $D$ always denote a digraph. From these adapted definitions we notice that there must be at least two directed paths to ensure that $v \leftrightarrow u$. It is then trivial, but important, to see the following relation described in Proposition \ref{prop:mut_set_in_D_and_G}

\begin{proposition}\label{prop:mut_set_in_D_and_G}
    If $S \subseteq D$ is a mutual-visibility set in $D$, and $G$ is the undirected graph of $D$, then $S$ is not necessarily a mutual-visibility set in $G$.
\end{proposition}

\begin{proof}
    Take a simple directed cycle $D$ as shown on the left and its undirected variant $G$ shown on the right in Figure \ref{fig:directed_graph_8_vertices} below. Notice that the mutual-visibility set $\{x,y,z\}$ in $D$ is not a mutual-visibility set in $G$:

    \begin{figure}[ht!]
    \centering
    \begin{tikzpicture}[
        scale=1,
        every node/.style={circle, draw, fill=black, minimum size=4pt, inner sep=1pt},
        font=\small,
        ->, 
        >=stealth 
    ]
        \begin{scope}
            \node[label=below:$x$] (x) at (0, 0) {};
            \node[label=above:$z$] (z) at (0.8, 0.5) {};
            \node[label=below:$y$] (y) at (1.6, 0) {};
            
            \def\radius{2.2cm}
            \def\centerx{0.8cm}
            \def\centery{0cm}
            
            \node (v1) at ({\centerx + \radius * cos(15)}, {\centery + \radius * sin(15)}) {};
            \node (v2) at ({\centerx + \radius * cos(45)}, {\centery + \radius * sin(45)}) {};
            \node (v3) at ({\centerx + \radius * cos(90)}, {\centery + \radius * sin(90)}) {};
            \node (v4) at ({\centerx + \radius * cos(135)}, {\centery + \radius * sin(135)}) {};
            \node (v5) at ({\centerx + \radius * cos(165)}, {\centery + \radius * sin(165)}) {};
            
            \draw[->, thick] (y) -- (z);
            
            \draw[<->, thick] (x) -- (z);
            \draw[<->, thick] (x) -- (v5);
            \draw[<->, thick] (v5) -- (v4);
            \draw[<->, thick] (v4) -- (v3);
            \draw[<->, thick] (v3) -- (v2);
            \draw[<->, thick] (v2) -- (v1);
            \draw[<->, thick] (v1) -- (y);
            
            \node[draw=none, fill=none] at (0.8, -1.2) {$D$};
        \end{scope}
        
        \begin{scope}[xshift=6cm]
            \node[label=below:$x$] (x2) at (0, 0) {};
            \node[label=above:$z$] (z2) at (0.8, 0.5) {};
            \node[label=below:$y$] (y2) at (1.6, 0) {};
            
            \def\radius{2.2cm}
            \def\centerx{0.8cm}
            \def\centery{0cm}
            
            \node (v12) at ({\centerx + \radius * cos(15)}, {\centery + \radius * sin(15)}) {};
            \node (v22) at ({\centerx + \radius * cos(45)}, {\centery + \radius * sin(45)}) {};
            \node (v32) at ({\centerx + \radius * cos(90)}, {\centery + \radius * sin(90)}) {};
            \node (v42) at ({\centerx + \radius * cos(135)}, {\centery + \radius * sin(135)}) {};
            \node (v52) at ({\centerx + \radius * cos(165)}, {\centery + \radius * sin(165)}) {};
            
            \draw[-, thick] (x2) -- (z2);
            \draw[-, thick] (y2) -- (z2);
            
            \draw[-] (x2) -- (v52);
            \draw[-] (v52) -- (v42);
            \draw[-] (v42) -- (v32);
            \draw[-] (v32) -- (v22);
            \draw[-] (v22) -- (v12);
            \draw[-] (v12) -- (y2);
            
            \node[draw=none, fill=none] at (0.8, -1.2) {$G$};
        \end{scope}
        \end{tikzpicture}
        \caption{Directed cycle with 8 vertices: $x$ and $y$ both direct to $z$, with remaining vertices forming a bidirectional path from $x$ to $y$. And, an undirected 8-cycle variant.}
        \label{fig:directed_graph_8_vertices}
    \end{figure}
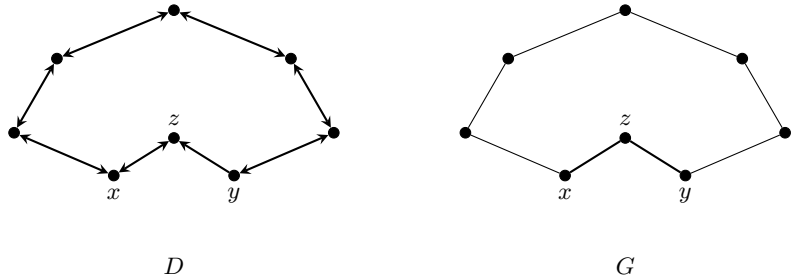
\end{proof}

Proposition \ref{prop:mut_set_in_D_and_G} shows that there isn't a simple translation between the mutual-visibility problem on undirected graph to directed graphs. This can be due to the fact that the shortest path from one vertex to another can differ depending on which vertex we start at, because of the orientation of the path. Hence, as seen in the example above, the directed path from $y$ to $z$ is different and shorter than the one from $z$ to $y$, which in this case gives more freedom for the size of the mutual-visibility set.

A natural place to start exploring the nature of $\mu(D)$ is within strongly connected components of digraphs. 

\subsection{Strongly connected components}

As per the definition of strongly connected components $\forall u,v \in S \subseteq D$, where $S$ is a strongly connected component, it holds that $u \to v$ and $v \to u$. Hence it is a trivial result that, if a digraph $D$ has a strongly connected component with size at least 2, then $\mu(D) \geq 2$. It would be interesting to observe, if the strongly connected components can also imply an upper bound on the mutual-visibility number. We define the well known condensation graph as follows:

\begin{definition}[Condensation graph]
    Let $SCC(D) = \{C_1, ..., C_k\}$ denote the set of strongly connected components of the digraph $D$. Then a condensation graph is defined as $D^{SCC}=(V^{SCC},E^{SCC})$, where
    \begin{equation}
        \begin{aligned}
        V^{SCC} =&\ SCC(D) \\ 
        E^{SCC} = \{e_{ij}\ |\  C_i,C_j \in SCC(D),&\ i \neq j,\ a \in C_i,\ b \in C_j : a \to b\}
        \end{aligned}
    \end{equation}
\end{definition}

\begin{remark}
    It is worth noting that since elements $SCC(D)$ do not intersect with each other, it is a partition of $V(D)$.
\end{remark}

Because of their structure condensation graphs are acyclic. This is implied from the fact that, if they were not, they would break the maximality of the underlying digraph's strongly connected components. We can use this fact to produce an upper bound on the mutual-visibility number of digraphs in relation to its strongly connected components.

\begin{theorem}
    \label{thm:mut_vis_in_one_scc}
    If there exists a mutual-visibility set $S$ in a digraph $D$ with $|S| \geq 2$ it must be entirely contained in one strongly connected component $C_i \in SCC(D)$. 
\end{theorem}

\begin{proof}
    Let $S\subseteq D$ be a mutual-visibility set in $D$ and let $D$ have at least two strongly connected components. Assume for a contradiction that $S$ intersects at least two distinct strongly connected components $C_1$ and $C_2$.

    Let $u\in S \cap C_1$ and $v \in S\cap C_2$. By the mutual-visibility property of $S$, an oriented cycle must exist, that contains both $u$ and $v$. This cycle implies that there exists a cycle in the condensation graph of $D$ which is impossible, hence such a cycle cannot exist, and the vertices $u$ and $v$ are not mutually-visible.
\end{proof}

\begin{corollary}
    The mutual-visibility number of a digraph $D$ is the maximum of the mutual-visibility numbers of its strongly connected components:
    \begin{equation}
        \mu(D) = \max_{C\in SCC(D)} \mu(C)
    \end{equation}
\end{corollary}

\begin{proof}
By Theorem \ref{thm:mut_vis_in_one_scc} any mutual-visibility set in $D$ must be contained in one strongly connected component, hence $\mu(D)$ is the largest $\mu(C)$ for any $C$ in $SCC(D)$. 
\end{proof}

An analogue of the cut edge from undirected graphs is the strong bridge in directed variants. We define a \textit{strong bridge} as the edge whose removal increases the number of strongly connected components.

\vspace{0.1cm}

If a strongly connected component $C$ contains a strong bridge $e$, and its removal adds at least two new components $X$ and $Y$ such that $X\to Y$ but $Y\not\to X$. If there is a mutual-visibility set $S$ in $C$ such that $v\in X \cap S$ and $u\in Y \cap S$, then the edge $e$ is critical to the path $u \to v$. This is a direct application of Menger's theorem to the directed graph case.

\begin{lemma}[Strong Bridge Bottleneck]
    \label{lem:strong_bridge_bottleneck}
    Let $D=(V,A)$ be a strongly connected digraph and let $e=(u,v)$ be a strong bridge. There exists a partition of the vertex set $(V_{source}, V_{sink})$ such that $e$ is the unique edge directed from $V_{source}$ to $V_{sink}$.
    
    Consequently, for any mutual-visibility set $S$ containing $x \in V_{source}$ and $y \in V_{sink}$, every path from $x$ to $y$ must contain the edge $e$.
\end{lemma}

\begin{proof}
    Since $e$ is a strong bridge, its removal yields a digraph $D' = D - e$ that is not strongly connected. This implies the existence of a cut $(V_{source}, V_{sink})$ such that there are no edges from $V_{source}$ to $V_{sink}$ in $D'$.
    
    In the original graph $D$, $e$ is therefore the unique edge crossing the cut from $V_{source}$ to $V_{sink}$. By Menger's Theorem, the maximum number of edge-disjoint paths from any $x \in V_{source}$ to any $y \in V_{sink}$ is equal to the size of the minimum edge cut separating them. Since $e$ is the only edge leaving $V_{source}$, any path starting in $V_{source}$ and ending in $V_{sink}$ must traverse $e$ to cross the cut. Thus, $e$ is part of every shortest $x,y$-path.
\end{proof}

What this lemma implies is that if there is a mutual-visibility set that contains a vertex from each set, the strong bridge blocks any other path between the two partitions, limiting the size of the set itself. In fact it provides a way of optimally choosing (or avoiding) vertices of a mutual-visibility set, as we formulate in the observation below:

\begin{corollary}
    \label{cor:bridge_endpoints}
    Let $e = (u,v)$ be a strong bridge separating partitions $V_{source}$ and $V_{sink}$. If a mutual-visibility set $S$ contains vertices from both partitions (i.e., it spans the bridge), then:
    \begin{enumerate}
        \item If $u \in S$, then $u$ must be the only vertex in $S \cap V_{source}$.
        \item If $v \in S$, then $v$ must be the only vertex in $S \cap V_{sink}$.
    \end{enumerate}
\end{corollary}

\begin{proof}
    Suppose $u \in S$ and there exists another vertex $z \in S \cap V_{source}$ such that $z \neq u$. Let $y$ be any vertex in $S \cap V_{sink}$.
    
    By Lemma \ref{lem:strong_bridge_bottleneck}, every path from $z$ to $y$ must traverse the edge $e = (u, v)$. Consequently, any such path is of the form $z \rightsquigarrow u \to v \rightsquigarrow y$. Since $u$ is an internal vertex of this path and $u \in S$, the vertex $u$ blocks the visibility between $z$ and $y$. This contradicts the assumption that $S$ is a mutual-visibility set.  The same logic applies if $v \in S$ for any pair traversing from $V_{source}$ to $V_{sink} \setminus \{v\}$.
\end{proof}

Although it is tempting to impose a \textit{linear upper bound} on the maximum mutual-visibility set using the strong bridge invariant, it turns out that the number of strong bridges is not correlated to the mutual-visibility number $\mu(D)$, as we prove next.

\begin{proposition}
    Let $\beta(D)$ denote the number of strong bridges in a strongly connected digraph $D$. There is no linear correlation between $\mu(D)$ and $\beta(D)$. specifically:
    \begin{enumerate}
        \item $\beta(D)$ can be arbitrarily large while $\mu(D)$ remains constant.
        \item $\mu(D)$ can be arbitrarily large while $\beta(D)$ remains constant.
    \end{enumerate}
\end{proposition}

\begin{proof}
    To prove the first case, consider the directed cycle $C_n$ with $n \geq 6$ vertices. Removing any edge in $C_n$ destroys strong connectivity, but increases the amount of (trivial) strongly connected components, so every edge is a strong bridge. Thus, $\beta(C_n) = n$. However, the mutual-visibility number of a directed cycle is constant, $\mu(C_n) = 2$. As $n \to \infty$, the number of bridges grows indefinitely while $\mu(D)$ does not.

    To prove the second case, consider a digraph $D$ constructed from two large complete digraphs (cliques) $K_n$ and $K'_n$, connected by two edges $e_1 = (u, v)$ and $e_2 = (v, u)$ where $u \in K_n$ and $v \in K'_n$. The edges $e_1$ and $e_2$ are the only edges connecting the two cliques; thus, removing either destroys the strong connectivity. Hence, $\beta(D) = 2$.
    
    However, within the cliques $K_n$ and $K'_n$, any subset of vertices forms a mutual-visibility set (as shortest paths are direct edges). We can select large subsets from $K_n$ and $K'_n$ (excluding the bridge endpoints $u,v$ per Corollary \ref{cor:bridge_endpoints}) to form a mutual-visibility set for $D$. Thus, as $n \to \infty$, $\mu(D)$ grows indefinitely while $\beta(D)$ remains fixed at 2.
    
    Since neither value bounds the other, they are not linearly correlated.
\end{proof}

We proved statements for the case where the digraph $D$ is one strongly connected component. In the general case, edges connecting distinct strongly connected components are not strong bridges, as their removal does not decompose any component. Furthermore, since Theorem \ref{thm:mut_vis_in_one_scc} restricts any mutual-visibility set to a single component, the 'bottleneck' analysis of Lemma \ref{lem:strong_bridge_bottleneck} need only be applied locally within each component. Thus, the problem of finding $\mu(D)$ in a general graph reduces to finding the maximum $\mu(C)$ among all components $C$, where each $\mu(C)$ is internally constrained by its own strong bridges.

\section{DAGs and Directed cycles}

In this section, we analyze the mutual-visibility number for directed acyclic and directed cyclic graph. The structural constraints of these classes allow us to determine exact values for $\mu(D)$.

\subsection{Directed Acyclic Graphs (DAGs)}

The definition of a mutual-visibility set requires that for any distinct pair $x, y \in S$, there must exist both a shortest path from $x$ to $y$ and a shortest path from $y$ to $x$. This bidirectional reachability condition leads to a trivial result for acyclic graphs.

\begin{proposition}
    For any directed acyclic graph $D$, the mutual-visibility number is $\mu(D) = 1$.
\end{proposition}

\begin{proof}
    By definition, a DAG contains no directed cycles. Consequently, for any distinct pair of vertices $u, v \in V(D)$, it is impossible to have both $u \rightarrow v$ and $v \rightarrow u$ (as this would form a cycle).
    
    According to Definition \ref{def:directed_mv_set}, a set $S$ with $|S| \geq 2$ requires at least one pair of distinct vertices to be mutually reachable. Since no such pair exists in a DAG, no set of size 2 or greater can satisfy the mutual-visibility property. Thus, the maximum size is limited to a single vertex (which vacuously satisfies the condition), implying $\mu(D) = 1$.
\end{proof}

\subsection{Directed Cycles}

Directed cycles ($C_n$) represent the simplest form of strong connectivity. Despite being strongly connected, their strict geometric structure severely limits mutual visibility.

\begin{proposition}
    For a directed cycle $C_n$ with $n \geq 3$, the mutual-visibility number is $\mu(C_n) = 2$.
\end{proposition}

\begin{proof}
    Let the vertices of $C_n$ be denoted $\{v_0, v_1, \dots, v_{n-1}\}$ with edges labeled as $(v_i, v_{i+1 \pmod n})$.
    
    \textit{Lower Bound ($\mu(C_n) \geq 2$):} 
    Consider any set of two distinct vertices $S = \{u, v\}$. The shortest path from $u$ to $v$ follows the cycle's orientation and contains only vertices from $V(C_n) \setminus \{u, v\}$ as internal vertices. Since $S$ contains no other vertices, these internal vertices are not in $S$. The same holds for the return path from $v$ to $u$. Thus, any pair is mutually visible.

    \textit{Upper Bound ($\mu(C_n) < 3$):} 
    Assume for the sake of contradiction that there exists a mutual-visibility set $S$ with $|S| \geq 3$. Let $x, y, z$ be three distinct vertices in $S$. Based on the cyclic ordering of $C_n$, one vertex must lie on the unique path between the other two. Without loss of generality, assume $y$ lies on the path from $x$ to $z$.
    
    The unique path from $x$ to $z$ is $P_{xz} = (x, \dots, y, \dots, z)$. Since $y \in S$ and $y$ is an internal vertex of $P_{xz}$, the visibility between $x$ and $z$ is blocked. This contradicts the definition of a mutual-visibility set.
    
    Therefore, no set of size 3 exists, and $\mu(C_n) = 2$.
\end{proof}

\section{Tournaments}

We now consider tournaments, which are directed graphs where every pair of distinct vertices is connected by exactly one directed edge. 

\begin{definition}[Tournament]
A tournament $T = (V, A)$ is a directed graph such that for every pair of distinct vertices $u, v \in V$, exactly one of $(u, v)$ or $(v, u)$ is in $A$.
\end{definition}

While directed cycles are strictly bounded by $\mu(C_n) = 2$, we show that strongly connected tournaments can support arbitrarily large mutual-visibility sets. This is due to the high density of edges, which facilitates the existence of "return paths" of length 2.

Historically, the properties of paths in tournaments have been studied via probabilistic methods. It is a standard result in the theory of random tournaments, discussed by Moon \cite{MOON1968}, that for sufficiently large $n$, every pair of vertices has approximately $n/4$ common outgoing neighbors. To provide a concrete lower bound for $\mu(T)$, we utilize \textit{Paley tournaments}, a class of deterministic tournaments that exhibit these "random-like" properties (often called quasirandom or doubly regular).

\begin{definition}[Paley Tournament]
Let $q$ be a prime power such that $q \equiv 3 \pmod 4$. The Paley tournament $P_q$ is a directed graph with vertex set $V = \mathbb{F}_q$, where there is an arc $(u, v)$ if and only if $v - u$ is a non-zero quadratic residue in $\mathbb{F}_q$.
\end{definition}

Intuitively, tournaments represent the structural opposite of DAGs regarding visibility. While DAGs enforce a strict ordering that prevents any return paths (yielding $\mu(D)=1$), tournaments are maximally connected in the sense that every pair has a direct relationship. The challenge for mutual visibility here lies in finding a "triangulated" return path of length 2. Paley tournaments, constructed from finite fields, possess quasirandom properties that ensure a dense distribution of these specific 2-step paths, effectively preventing any small set of vertices from blocking all return routes.

\begin{theorem}
For any integer $k \geq 1$, there exists a tournament $T$ such that $\mu(T) \geq k$. Specifically, if $q$ is a prime power with $q \equiv 3 \pmod 4$ and $q > 4k - 5$, then $\mu(P_q) \geq k$.
\end{theorem}

\begin{proof}
Let $T = P_q$ satisfy the condition $q > 4k - 5$. Let $S \subseteq V(T)$ be any set of $k$ vertices. We show that $S$ is a mutual-visibility set.

Consider any distinct pair $u, v \in S$. By the definition of a tournament, exactly one arc exists between them. Without loss of generality, assume $(u, v) \in A$. This arc represents a shortest path of length 1. Since paths of length 1 contain no internal vertices, the visibility condition from $u$ to $v$ is trivially satisfied ($S \cap \emptyset = \emptyset$).

For $u$ and $v$ to be mutually visible, there must also be a shortest path from $v$ to $u$ whose internal vertices avoid $S$. Since $(u, v) \in A$, the shortest possible path from $v$ to $u$ must have length at least 2. Such a path exists if there is a vertex $w$ such that $(v, w) \in A$ and $(w, u) \in A$.

A fundamental property of Paley tournaments $P_q$ is that they are \textit{doubly regular}: for any pair of distinct vertices $u, v$, the number of common neighbors $N(v, u) = \{w \in V : v \to w \text{ and } w \to u\}$ is constant:
\[ |N(v, u)| = \frac{q-3}{4} \]
To ensure that $u$ and $v$ are mutually visible, we require at least one vertex in $N(v, u)$ to be outside the set $S$. The set $S$ contains $k$ vertices, two of which are $u$ and $v$. Thus, the number of "blocking" candidates in $S$ is $|S \setminus \{u, v\}| = k - 2$.

A valid internal vertex $w \notin S$ exists if:
\[ |N(v, u)| > k - 2 \]
Substituting the Paley property:
\[ \frac{q-3}{4} > k - 2 \]
\[ q - 3 > 4k - 8 \]
\[ q > 4k - 5 \]
Since we chose $q$ to satisfy this inequality, there always exists a path $v \to w \to u$ with $w \notin S$. Thus, every pair in $S$ is mutually visible, and $\mu(P_q) \geq k$.
\end{proof}

\noindent
This result shows that for any $k\geq 1$, there exists such a tournament, in particular a Paley tournament $P_q$, such that $\mu({P_q}) \geq k$. 

\section{Algorithmic Complexity}

We continue by addressing the computational complexity of the mutual-visibility problem in directed graphs. Following the foundational work of Di Stefano \cite{DISTEFANO2022126850}, we analyze both the complexity of verifying a candidate set and the hardness of determining the mutual-visibility number $\mu(D)$.

\subsection{Verification Algorithm}
The \textit{Mutual-Visibility Verification Problem} asks whether a given subset $S \subseteq V(D)$ is a mutual-visibility set. Unlike the optimization problem, verification can be performed efficiently in polynomial time.

\begin{proposition}
    Given a directed graph $D=(V, A)$ and a subset $S \subseteq V$, the validity of $S$ as a mutual-visibility set can be verified in $O(|S| \cdot (|A| + |V| \log |V|))$ time using Dijkstra's algorithm, or $O(|S| \cdot (|V| + |A|))$ for unweighted graphs.
\end{proposition}

\begin{proof}
    To verify $S$, we must check that for every distinct ordered pair $(u, v) \in S \times S$, there exists a shortest $u,v$-path in $D$ that contains no vertices from $S \setminus \{u, v\}$.
    
    The algorithm proceeds as follows:
    \begin{enumerate}
        \item \textbf{Global Distances:} First, compute the shortest path distances between all pairs in $S$ within the original graph $D$. Let $d_D(u, v)$ denote this distance. This can be done by running $|S|$ Single-Source Shortest Path (SSSP) computations (e.g., BFS or Dijkstra) starting from each $s \in S$.
        \item \textbf{Restricted Distances:} For each source $u \in S$, compute the shortest path distances to all other $v \in S$ in the restricted graph $D' = D[V \setminus (S \setminus \{u\})]$. Note that we remove all vertices of $S$ except the source $u$. Let $d_{restricted}(u, v)$ be the distance in this subgraph.
        \item \textbf{Comparison:} For each pair $(u, v)$, $u$ and $v$ are mutually visible if and only if $d_{restricted}(u, v) = d_D(u, v)$ and $d_{restricted}(v, u) = d_D(v, u)$. 
    \end{enumerate}
    
    If $d_{restricted}(u, v) > d_D(u, v)$, it implies that all original shortest paths passed through some vertex in $S \setminus \{u, v\}$, thus blocking visibility. The cost is dominated by $|S|$ executions of BFS/Dijkstra, yielding the stated time complexity.
\end{proof}

\noindent
A formal pseudo code of the algorithm is given below in Algorithm \ref{alg:verification}.

\begin{algorithm}
\caption{Mutual-Visibility Set Verification}
\label{alg:verification}
\begin{algorithmic}[1]
\Require Digraph $D=(V, A)$, Subset $S \subseteq V$
\Ensure \textbf{True} if $S$ is a mutual-visibility set, \textbf{False} otherwise

\State \Comment{\textbf{Step 1: Compute Global Distances}}
\For{each vertex $u \in S$}
    \State Run BFS starting from $u$ in $D$ to compute shortest path distances $d_D(u, v)$ for all $v \in S$.
\EndFor

\State \Comment{\textbf{Step 2: Compute Restricted Distances}}
\For{each vertex $u \in S$}
    \State Initialize restricted distances $d_{restr}[v] \gets \infty$ for all $v \in V$
    \State $d_{restr}[u] \gets 0$
    \State Initialize Queue $Q$, $Q.enqueue(u)$
    
    \While{$Q$ is not empty}
        \State $curr \gets Q.dequeue()$
        
        \State \Comment{If we hit another node in S, we record distance but do not pass through}
        \If{$curr \in S$ \textbf{and} $curr \neq u$}
            \State \textbf{continue} 
        \EndIf
        
        \For{each outgoing neighbor $w$ of $curr$}
            \If{$d_{restr}[w] = \infty$}
                \State $d_{restr}[w] \gets d_{restr}[curr] + 1$
                \State $Q.enqueue(w)$
            \EndIf
        \EndFor
    \EndWhile
    \State Store $d'_{u}(v) \gets d_{restr}[v]$ for all $v \in S$
\EndFor

\State \Comment{\textbf{Step 3: Verification}}
\For{each pair of distinct vertices $u, v \in S$}
    \If{$d'_{u}(v) \neq d_D(u, v)$ \textbf{or} $d'_{v}(u) \neq d_D(v, u)$}
        \State \Return \textbf{False} \Comment{A shortest path is blocked by S}
    \EndIf
\EndFor

\State \Return \textbf{True}
\end{algorithmic}
\end{algorithm}

\subsection{Hardness of Finding $\mu(D)$}
While verification is efficient, finding the maximum size of such a set is computationally intractable. Di Stefano proved that the mutual-visibility problem is NP-complete for general undirected graphs \cite{DISTEFANO2022126850}. We extend this result to directed graphs via a simple reduction.

\begin{theorem}
    The problem of computing the mutual-visibility number $\mu(D)$ for a general digraph $D$ is NP-hard.
\end{theorem}

\begin{proof}
    To prove hardness, we employ a reduction from the undirected version of the Mutual-Visibility problem, established as NP-complete by Di Stefano \cite{DISTEFANO2022126850}. 
    
    Let $G=(V, E)$ be an instance of the undirected problem. We construct a directed instance $D=(V, A)$ by replacing every undirected edge $\{u, v\} \in E$ with two symmetric directed edges $(u, v)$ and $(v, u)$.
    
    In this symmetric digraph $D$, a path $u \to v$ exists if and only if a path exists in $G$, and their lengths are identical. Furthermore, if a set $S$ is a mutual-visibility set in $G$, shortest paths between any $u, v \in S$ avoid $S \setminus \{u, v\}$. The corresponding paths in $D$ use the same sequence of vertices and thus satisfy the directed mutual-visibility condition for both $u \to v$ and $v \to u$.
    
    Consequently, $\mu(D) = \mu(G)$. Since determining $\mu(G)$ is NP-complete, determining $\mu(D)$ must also be NP-hard.
\end{proof}

\section{Conclusion and Future Work}

In this paper, we formalized the mutual-visibility problem for directed graphs, adapting the concept from the undirected literature. While we introduced definitions for several variants—including \textbf{total}, \textbf{outer}, and \textbf{dual} mutual visibility—our analysis in this work focused primarily on the fundamental \textit{mutual-visibility number} $\mu(D)$. 

For this standard invariant, we established bounds and constructions of the bounds across fundamental graph classes; DAGs, directed cycles, and tournaments.

We further proved that finding the maximum mutual-visibility set is NP-hard for general digraphs, though verification remains polynomial, which is in fact equivalent to the undirected mutual-visibility case.

\vspace{5mm}

As this work represents the first step into directed mutual visibility, several promising avenues for future research remain:

\paragraph{Analysis of Defined Variants}
In Section \ref{sec:mut_visibilities}, we formalized the concepts of \textbf{total}, \textbf{outer}, and \textbf{dual} mutual-visibility sets for digraphs. These variants impose stricter conditions, such as requiring visibility between vertices inside and outside the set, or among vertices entirely outside the set. While we established their definitions to provide a complete framework, we left their specific properties, bounds, and complexity landscapes unexplored. A natural next step is to extend our analysis of cycles and tournaments to these stricter variants to see if similar hierarchies exist.

\paragraph{Semi-Directed Mutual Visibility}
Our current definition requires a valid shortest path in \textit{both} directions ($u \to v$ and $v \to u$). A significant relaxation would be to require visibility in only \textit{one} direction (e.g., $u \to v$) while simply requiring reachability in the other. This "semi-directed" variant may be more applicable to asynchronous communication networks where data flow does not need to be simultaneous.

\paragraph{The Condensation Graph Hierarchy}
While we showed that a mutual-visibility set must reside within a single Strongly Connected Component (SCC), the interplay between the "internal" visibility of an SCC and its position in the condensation graph warrants further study. Specifically, determining if the mutual-visibility properties of a "parent" SCC influence the effective visibility of a "child" SCC in less restrictive variants of the problem remains an open question.

\bibliographystyle{elsarticle-num} 
\bibliography{references}

@article{DISTEFANO2022126850,
title = {Mutual visibility in graphs},
journal = {Applied Mathematics and Computation},
volume = {419},
pages = {126850},
year = {2022},
issn = {0096-3003},
doi = {https://doi.org/10.1016/j.amc.2021.126850},
url = {https://www.sciencedirect.com/science/article/pii/S0096300321009334},
author = {Gabriele {Di Stefano}}
}

@article{CICERONE2023114096,
title = {Variety of mutual-visibility problems in graphs},
journal = {Theoretical Computer Science},
volume = {974},
pages = {114096},
year = {2023},
issn = {0304-3975},
doi = {https://doi.org/10.1016/j.tcs.2023.114096},
url = {https://www.sciencedirect.com/science/article/pii/S0304397523004097},
author = {Serafino Cicerone and Gabriele {Di Stefano} and Lara Drožđek and Jaka Hedžet and Sandi Klavžar and Ismael G. Yero}
}

@book{MOON1968,
  title={Topics on Tournaments},
  author={Moon, John W.},
  year={1968},
  publisher={Holt, Rinehart and Winston},
  address={New York},
  note={See Chapter 1 for cycle properties and Chapter 8 for properties of regular tournaments.}
}

@misc{roy2025vertexvisibilitynumbergraphs,
      title={The vertex visibility number of graphs}, 
      author={Dhanya Roy and Gabriele Di Stefano and Sandi Klavžar and Aparna Lakshmanan S},
      year={2025},
      eprint={2510.19452},
      archivePrefix={arXiv},
      primaryClass={cs.DM},
      url={https://arxiv.org/abs/2510.19452}, 
}
\end{document}